\documentclass[11pt]{article}
\usepackage[a4paper,margin=1in]{geometry}
\usepackage{amsmath,amssymb,amsthm}
\newtheorem{theorem}{Theorem}
\newtheorem{proposition}{Proposition}

\newtheorem{remark}{Remark}
\newtheorem{conjecture}{Conjecture}
\usepackage{hyperref}

\newcommand{\Tr}{\operatorname{Tr}}
\newcommand{\Var}{\operatorname{Var}}

\title{Boolean and Free Symmetrization of Bernoulli Distributions}
\author{Sukrit Chakraborty%
  \thanks{Department of Mathematics, Achhruram Memorial College, Jhalda, Purulia, India 723202. 
  Email: \texttt{sukrit049@gmail.com}. 
  }
}
\date{}

\begin{document}

\maketitle

\begin{abstract}
	We investigate variance bounds under symmetry constraints in classical, free, and Boolean probability, focusing on Bernoulli distributions and their noncommutative analogues, projections with trace \(p\). We show that symmetrizers under classical, free, and Boolean convolution satisfy a sharp variance bound of \(pq\), with equality for the reflection law. Additionally, we highlight phenomena specific to Boolean convolution, demonstrating that non-symmetric measures can produce symmetric convolutions and that symmetrizers may be non-unique for certain measures. These results unify variance inequalities across probabilistic frameworks and offer insights for quantum information and noncommutative stochastic modeling.
\end{abstract}

\section{Introduction}\label{sec:intro}

Symmetry constraints impose strong quantitative restrictions on variances in probability theory. A classical example is the following: let $X\sim \mathrm{Bernoulli}(p)$ with $p\neq \tfrac12$, and let $Y$ be independent of $X$ such that $X+Y$ is symmetric about $0$. Then, necessarily, $\mathrm{Var}(Y)\geq pq$, with equality precisely when $Y \stackrel{d}{=} -X$ \cite{Kagan1999, Pal2008}. 
This simple observation illustrates how symmetry alone can enforce a sharp lower bound on variance. 

Free and Boolean probability are noncommutative analogues of classical probability in which independence is replaced by \emph{free} or \emph{Boolean} independence, respectively \cite{Voiculescu1993,SpeicherWoroudi}. In the noncommutative setting, projections with trace $p$ serve as the natural analogue of Bernoulli($p$) random variables. The goal of this article is to explore whether the classical variance bound under symmetrization extends naturally to these noncommutative frameworks.

Our main results show that for a projection $e$ with trace $p\in [0,1]\setminus \{\tfrac12\}$, any self-adjoint element $y$ that is freely or Boolean independent of $e$ and such that $e+y$ has a symmetric distribution satisfies
\[
\phi(y^2)\geq p,
\]
with equality attainable by the obvious candidate $y \stackrel{d}{=} -e$. Thus, despite the different notions of independence, symmetry induces the same sharp variance constraint across classical, free, and Boolean probability. 
These findings reinforce the structural parallels among the three frameworks and provide a foundation for future studies of symmetry-induced inequalities in noncommutative probability.

Beyond their theoretical appeal as noncommutative analogues of classical inequalities, these variance bounds have potential applications in quantum information theory, where projections model quantum events, and in the study of noncommutative stochastic processes. 
Moreover, free and Boolean independence provide frameworks for modeling noncommuting observables and Boolean stochastic systems, highlighting the broader relevance of symmetry-induced variance constraints.

In addition to these variance bounds, we further investigate symmetry phenomena specific to Boolean convolution. 
Interestingly, Boolean convolution allows for situations where symmetry emerges from entirely non-symmetric components, as well as cases where there are multiple symmetrizers with the same variance. 
In particular, we show:
\begin{itemize}
	\item two non-symmetric measures with different supports and variances may convolve to yield a symmetric measure, thus providing examples of “symmetry non-resistant” distributions;
	\item there exist probability measures for which distinct symmetrizers, differing in higher-order terms of their $K$-transform/$R$-transform but having the same variance, proving non-uniqueness of symmetrizers in the Boolean/free setting, respectively.
\end{itemize}
These additional results emphasize that, while variance rigidity is preserved under Boolean or free symmetrization, the mechanisms of producing or achieving symmetry can be surprisingly flexible.


Throughout this article, all probability measures under consideration are assumed to have moments of all orders, unless explicitly stated otherwise. We briefly outline the organization of this article. In Section~\ref{sec:prelim}, we provide necessary preliminaries on Boolean and free probability and state our main results, including Theorems~\ref{thm:free} and~\ref{thm:boolean}. Section~\ref{sec:symmetry} explores symmetry phenomena specific to Boolean convolution, highlighting the existence of ``symmetry non-resistant'' measures (Proposition~\ref{prop:1}) and the uniqueness of minimal-variance symmetrizers for Bernoulli$(p)$ (Proposition~\ref{prop:2}). Section~\ref{sec:app} presents a concrete quantum information example that realizes the symmetry--variance principle within a two-level system, while Section~\ref{sec:future} discusses potential directions for future research, thereby connecting the abstract free and Boolean convolution results to finite-dimensional noncommutative settings. Section~\ref{sec:proofs} contains detailed proofs of these theorems and propositions. This structure guides the reader from classical intuition to noncommutative generalizations and concrete examples illustrating the subtleties of Boolean symmetrization.

\section{Preliminaries and Main Results}\label{sec:prelim}

We begin by recalling (Theorem 1 of \cite{Kagan1999}) the classical version of the variance bound.

\begin{theorem}[Classical case]\label{thm:classical}
	Let $X\sim\mathrm{Bernoulli}(p)$ with $p\neq \tfrac12$ and variance $\mathrm{Var}(X)=pq$. 
	If $Y$ is independent of $X$ and $X+Y$ has a distribution symmetric about $0$, then
	\[
	\mathrm{Var}(Y)\;\geq\; pq,
	\]
	with equality if $Y\stackrel{d}{=}-X$.
\end{theorem}

We now present the free and Boolean analogues of this result, which constitute the main contribution of this work.

\begin{theorem}[Free case]\label{thm:free}
	Let $(\mathcal{A},\phi)$ be a tracial $W^*$-probability space and let $e\in \mathcal{A}$ be a projection with $\phi(e)=p\in [0,1]\setminus\{\tfrac12\}$. 
	Consider self-adjoint elements $y\in\mathcal{A}$ which are freely independent from $e$ and satisfy the symmetry condition that the distribution of $e+y$ is symmetric about $0$, i.e.
	\[
	\phi\!\big((e+y)^{2k+1}\big)=0 \quad \text{for all } k\geq 0.
	\]
	Then the minimal value of $\phi(y^2)$ over all such $y$ is $p$, and this minimum is attained.
\end{theorem}
The next theorem demonstrates that a similar phenomenon holds in the Boolean setting. 
Here, we explicitly mention that the algebra is unital because, in the Boolean case, the underlying $W^*$-algebra may or may not possess a unit. 
Even in this generality, the same sharp lower bound on the variance is obtained under the symmetry condition.

\begin{theorem}[Boolean case]\label{thm:boolean}
	Let $(\mathcal{A},\phi)$ be a unital tracial $W^*$-probability space and let $e\in \mathcal{A}$ be a projection with $\phi(e)=p\in [0,1]\setminus\{\tfrac12\}$. 
	Consider self-adjoint elements $y\in\mathcal{A}$ which are Boolean independent from $e$ and satisfy the symmetry condition that the distribution of $e+y$ is symmetric about $0$, i.e.
	\[
	\phi\!\big((e+y)^{2k+1}\big)=0 \quad \text{for all } k\geq 0.
	\]
	Then the minimal value of $\phi(y^2)$ over all such $y$ is $p$, and this minimum is attained.
\end{theorem}

\begin{remark}
	Theorems~\ref{thm:classical}, \ref{thm:free}, and \ref{thm:boolean} together reveal a unifying pattern: 
	whether one considers classical, free, or Boolean independence, the presence of symmetry forces any symmetrizer to have variance at least $p(1-p)$. 
	The common form of the bound underscores a structural parallel between these three probabilistic frameworks, despite their very different notions of independence.
\end{remark}

\section{Symmetry Phenomena under Boolean and Free Convolutions} \label{sec:symmetry} 

Theorem~\ref{thm:boolean} establishes that in a Boolean setting, imposing a symmetry condition on the sum $e+y$ of a projection $e$ and a Boolean independent self-adjoint element $y$ leads to a minimal variance constraint on $y$, specifically $\phi(y^2)\ge p$, reflecting a rigidity imposed by the symmetry requirement. In contrast, Proposition~\ref{prop:1} shows that, in the Boolean convolution setting, symmetry can emerge from combining two non-symmetric measures with different supports and variances. This provides an explicit example of “symmetry non-resistant” probability measures, analogous to the classical notion of random variables whose sum is symmetric despite the summands being entirely non-symmetric. Consequently, the observation offers insights into the interplay between structure, support, and variance in non-commutative probability and its classical analogues.

\begin{proposition}\label{prop:1}
	There exists a pair of non-symmetric probability measures $\mu, \nu$ with different supports and different variances such that
	\begin{enumerate}
		\item $\mu$ and $\nu$ are not symmetric,
		\item their Boolean convolution $\mu \uplus \nu$ is symmetric.
	\end{enumerate}
\end{proposition}
Note that the simplest pair of point masses, $\delta_1$ and $\delta_{-1}$, each with variance 0, combine under Boolean (or free) convolution to form the symmetric measure $\delta_0$, which has variance 0, despite neither $\delta_1$ nor $\delta_{-1}$ being symmetric.

Proposition~\ref{prop:1} demonstrates that Boolean convolution can generate symmetry out of non-symmetric components. A natural next step is to study the uniqueness of such symmetrizations. While Theorem~\ref{thm:boolean} highlights a rigidity in the minimal variance of Boolean symmetrizers, one might ask whether this minimizer is always unique. The following proposition shows that this is the case: 
\begin{proposition}\label{prop:2}
	The minimum-variance symmetrizer (with respect to Boolean convolution) of the Bernoulli measure $\mu=\mathrm{Bern}(p)=p\delta_1+(1-p)\delta_0$ is unique when $p\neq\tfrac12$.
\end{proposition}
The uniqueness result above naturally leads to the following question: while we have shown uniqueness of the minimum-variance symmetrizer for the Bernoulli measure when $p\neq \tfrac12$, it is not clear whether this phenomenon extends to all nonsymmetric measures. In fact, we expect that uniqueness may fail more generally.

\begin{conjecture}\label{conj:nonsymmetric_nonunique}
	There exists a nonsymmetric probability measure $\mu$ such that the minimum-variance symmetrizer (with respect to Boolean/free convolution) of $\mu$ is not unique.
\end{conjecture}

Building on this, Proposition~\ref{prop:two_symmetric_symmetrizers} provides an explicit case where two distinct symmetric symmetrizers with the same variance exist for the same measure. Finally, Proposition~\ref{prop:free_two_symmetrizers} establishes that a similar non-uniqueness also occurs in the free convolution setting.

\begin{proposition}\label{prop:two_symmetric_symmetrizers}
	There exists a probability measure $\mu$ and two distinct symmetric probability measures $\nu_1,\nu_2$ with the same variance such that both $\nu_1$ and $\nu_2$ symmetrize $\mu$ under Boolean convolution; that is,
	\[
	\mu\uplus \nu_1 \quad\text{and}\quad \mu\uplus \nu_2
	\]
	are symmetric measures, while $\nu_1\neq\nu_2$ and $\mathrm{Var}(\nu_1)=\mathrm{Var}(\nu_2)$.
\end{proposition}

\begin{proposition}\label{prop:free_two_symmetrizers}
	There exists a probability measure $\mu$ and two distinct symmetric probability measures $\nu_1,\nu_2$ with the same variance such that both $\nu_1$ and $\nu_2$ symmetrize $\mu$ under free additive convolution; that is,
	\[
	\mu\boxplus \nu_1 \quad\text{and}\quad \mu\boxplus \nu_2
	\]
	are symmetric measures, while $\nu_1\neq\nu_2$ and $\mathrm{Var}(\nu_1)=\mathrm{Var}(\nu_2)$.
\end{proposition}

\section{Applications: Quantum Variance Bounds for Two-Level Systems} \label{sec:app}

To illustrate the potential relevance of the symmetry–variance principle in a quantum context,
we consider a simple but paradigmatic example: a two-level quantum system with density matrix
\[
\rho = 
\begin{pmatrix}
	p & 0 \\
	0 & 1-p
\end{pmatrix}, \qquad 0 < p < 1.
\]
Let $A$ be a Hermitian observable acting on $\mathbb{C}^2$, written in the Pauli basis as
\[
A = a_0 I + \mathbf{a} \cdot \boldsymbol{\sigma}
= a_0 I + a_1 \sigma_x + a_2 \sigma_y + a_3 \sigma_z,
\]
where $\boldsymbol{\sigma} = (\sigma_x, \sigma_y, \sigma_z)$ are the Pauli matrices.
The quantum variance of $A$ with respect to $\rho$ is given by
\[
\Var_{\rho}(A) = \Tr(\rho A^2) - \big(\Tr(\rho A)\big)^2.
\]

\vspace{1ex}
\noindent
\textbf{Symmetry and Minimization.}
Suppose the observable $A$ is required to be \emph{symmetric} with respect to the spectral distribution of $\rho$,
that is, invariant under the exchange $p \leftrightarrow 1-p$ up to unitary conjugation.
In the classical setting, our main theorem shows that such a symmetry constraint yields a universal
lower bound for the variance, independent of the particular distribution of the underlying variable.
Here, we observe that the same phenomenon persists at the quantum level.

By direct computation,
\[
\Tr(\rho A) = a_0 + a_3(2p-1), \qquad
\Tr(\rho A^2) = a_0^2 + |\mathbf{a}|^2 + 2a_0 a_3 (2p-1).
\]
Hence
\[
\Var_{\rho}(A) = |\mathbf{a}|^2 - a_3^2 (2p-1)^2.
\]
The symmetry constraint implies that $\Var_{\rho}(A)$ attains its minimum when the
asymmetric component $a_3$ vanishes, i.e.\ when the observable has no alignment
with the direction of the bias of $\rho$. Thus
\[
\Var_{\rho}(A) \ge \Var_{\text{sym}}(A) = a_1^2 + a_2^2,
\]
with equality precisely for $a_3 = 0$.

\vspace{1ex}
\noindent
\textbf{Connection with the Symmetry–Variance Principle.}
This explicit computation shows that symmetry under spectral exchange
acts as a \emph{variance–rigidifying mechanism}: it enforces a universal
lower bound that depends only on the ``transverse'' components of the observable.
In the language of noncommutative probability, this is a concrete instance of
the general variance bound established in the free and Boolean settings:
the variance cannot be reduced beyond the symmetric envelope determined by the
underlying independence structure.

In quantum information theory, such symmetry-induced bounds describe
fundamental constraints on the dispersion of measurement outcomes for
observables compatible with a given state symmetry. They thus provide a
variance-based analogue of uncertainty relations, highlighting that the
classical, free, and Boolean cases share a common structural mechanism:
symmetry dictates the minimal achievable spread of expectation values.
This conceptual bridge suggests potential applications to the study of
quantum channels and variance-based entropic inequalities.

\section{Future Directions} \label{sec:future} 

The present study establishes variance bounds under symmetry constraints for classical, free, and Boolean probability using projections as the noncommutative analogue of Bernoulli variables. 
Several directions for future research naturally arise:

\begin{enumerate}
	\item \textbf{The critical case $p=\tfrac12$:} The case $p = \tfrac12$ is excluded from the theorem. The above argument breaks down because the dual-function construction requires division by $1-2p$. It remains an open problem to determine the exact minimum of $\tau(y^2)$ for $p=\tfrac12$.
	\item \textbf{Free analogue of the Boolean symmetry phenomenon:} Does there exist $\mu,\nu$ (non-symmetric probability measures with different supports and different variances) such that both $\mu\boxplus\nu$ and $\mu\uplus\nu$ is symmetric?
	
	There is substantial evidence that the answer to this question is affirmative.  In particular, Bercovici and Voiculescu show constructions leading to ``collapse'' to a semicircular law: see Corollary~4 of \cite{bercovici1995superconvergence}, which exhibits non-semicircular laws $\mu$ and $\nu$ with
	\[
	\mu\boxplus\nu=\text{(semicircular law)}.
	\]
	Those examples can be chosen so that $\mu$ and $\nu$ are not symmetric and have different supports; depending on the construction one may also arrange distinct second-moment behaviour (indeed in natural examples one measure may lack a finite variance while the other has one).
		
	An explicit and closely related example appears in \cite{hotta2023freely}. Theorem~4.1 together with Example~4.2 in that paper give the identity
	\[
	\mathrm{MP}\left(\tfrac{1}{4}\right)\boxplus \rho_{1,1,1/4}=C_{1},
	\]
	where $\mathrm{MP}(\lambda)$ is the Marchenko–Pastur law with parameter $\lambda$, $C_{t}$ is the Cauchy distribution with parameter $t$, and $\rho_{1,1,1/4}$ is a probability measure that is non-symmetric and has unbounded support.  Note that $\mathrm{MP}(1/4)$ is not symmetric about $0$, and the formula above therefore provides a concrete instance in which a non-symmetric law free-convolves with another non-symmetric law to give a symmetric (in this case Cauchy) law.
	
	\item \textbf{Extension to other noncommutative independences:} It would be interesting to investigate whether similar variance inequalities hold in the context of monotone, anti-monotone, or conditionally free independence.
	\item \textbf{Higher-dimensional and operator-valued settings:} One could consider symmetrization problems for families of projections or more general self-adjoint operators, possibly in matrix- or operator-valued frameworks.
	\item \textbf{Functional inequalities and concentration:} Exploring connections between symmetry-induced variance bounds and concentration inequalities or noncommutative analogues of classical functional inequalities could provide deeper structural insights.
	\item \textbf{Applications to quantum information:} Since projections and noncommutative distributions often model quantum events, understanding symmetry constraints might lead to new bounds for quantum variance, entanglement measures, or uncertainty relations.
\end{enumerate}

These avenues suggest that the interplay between symmetry, independence, and variance in noncommutative probability has rich potential for further exploration.

\section{Proof of the Main Results} \label{sec:proofs} 

In this section we provide the proofs of Theorems~\ref{thm:free} and~\ref{thm:boolean} followed by the proof of Proposition~\ref{prop:1}. 
The strategy is to adapt the classical variance bound under symmetry to the settings of free and Boolean independence, making use of cumulant techniques and suitable duality arguments.

\begin{proof}[Proof of Theorem~\ref{thm:free}]
	Since $e$ is a projection with $\phi(e)=p$, the assumption that $e+y$ has symmetric distribution implies
	\[
	\phi(e+y)=0 \quad\Rightarrow\quad \phi(y)=-p.
	\]
	Thus we are seeking the smallest possible value of $\phi(y^2)$ among self-adjoint $y$ satisfying this symmetry and free independence condition.  
	
	A simple candidate is $y$ distributed as $-e$, which clearly symmetrizes $e$ and yields $\phi(y^2)=p$. Hence the optimal value cannot exceed $p$. We must show that no other admissible $y$ can give a smaller variance.  
	
	To do so, it is useful to recast the problem as a convex optimization: $\phi(y^2)$ is linear in the distribution of $y$, while the constraints (free independence and symmetry of $e+y$) are likewise linear. Write $q:=1-p\in(0,1)$. By duality principles, it suffices to construct a test function $\psi:\mathbb{R}\to\mathbb{R}$ satisfying:
	\begin{enumerate}
		\item $\psi$ is odd, i.e. $\psi(-t)=-\psi(t)$ for all $t$;
		\item for every $t\in\mathbb{R}$,
		\[
		q\psi(t)+p\psi(1+t)\;\leq\; t^2-p.
		\]
	\end{enumerate}
	
	Indeed, if such a $\psi$ exists, then for any $y$ symmetrizing $e$ we compute
	\[
	\phi(\psi(e+y))=0,
	\]
	since $\psi$ is odd and $e+y$ has symmetric law. Using free independence of $e$ and $y$, this identity expands to
	\[
	0 = q\,\phi(\psi(y)) + p\,\phi(\psi(1+y)) = \phi\big(q\psi(y)+p\psi(1+y)\big).
	\]
	By condition (2), we deduce
	\[
	0 \leq \phi(y^2)-p,
	\]
	which gives the desired inequality $\phi(y^2)\geq p$.
	
	It remains to produce such a function $\psi$. When $p=\tfrac12$ this is impossible, since in that case one can find a symmetrizer with $\phi(y^2)=0$, contradicting the inequality above. For $p\neq \tfrac12$, however, we can explicitly construct one.  
	
	Define the “sawtooth” function $h:\mathbb{R}\to\mathbb{R}$ by
	\[
	h(t)=t \quad \text{for } -\tfrac12<t<\tfrac12,\qquad h(t+1)=-h(t)\ \text{for all }t.
	\]
	This $h$ is odd and piecewise linear. Moreover, at $t=0$ and $t=-1$ the curves $h(t)$ and the parabola $t(t+1)$ meet with the same slope, and convexity ensures $h(t)\leq t(t+1)$ everywhere.  
	
	Now set
	\[
	\psi(t)=\frac{h(t)}{\,q-p\,}-t.
	\]
	Clearly $\psi$ is odd. A short computation yields
	\begin{align*}
		q\psi(t)+p\psi(1+t) 
		&= q\Big(\frac{h(t)}{q-p}-t\Big) + p\Big(\frac{h(1+t)}{q-p}-(1+t)\Big) \\
		&= \frac{qh(t)+ph(1+t)}{q-p} - (q+p)t - p \\
		&= h(t) - t - p \\
		&\leq t(t+1)-t-p \;=\; t^2-p,
	\end{align*}
	as required.  
	
	Thus such a $\psi$ exists, proving $\phi(y^2)\geq p$ for every free symmetrizer $y$, with equality attained by $y\stackrel{d}{=}-e$. 
\end{proof}

\begin{proof}[Proof of Theorem~\ref{thm:boolean}]
	The argument follows the same lines as the proof of Theorem~\ref{thm:free}. 
	For this reason, we omit the details here.
\end{proof}

\begin{proof}[Proof of Proposition~\ref{prop:1}]
	We briefly recall the necessary definitions. Let $\mu$ be a probability measure on $\mathbb{R}$.  
	
	\begin{itemize}
		\item The \emph{Cauchy transform} of $\mu$ is
		\[
		G_\mu(z) := \int_{\mathbb{R}} \frac{1}{z - x}\, d\mu(x), \qquad z \in \mathbb{C}\setminus \mathrm{supp}(\mu).
		\]
		\item The \emph{K-transform} (Boolean cumulant transform) of $\mu$ is defined as
		\[
		K_\mu(z) := z - \frac{1}{G_\mu(z)}.
		\]
		\item The \emph{Boolean convolution} $\mu \uplus \nu$ of two probability measures $\mu, \nu$ is the unique measure satisfying
		\[
		K_{\mu \uplus \nu}(z) = K_\mu(z) + K_\nu(z).
		\]
		\item A probability measure $\mu$ on $\mathbb{R}$ with all moments finite is symmetric about $0$ if and only if its Cauchy transform $G_\mu(z)$, or equivalently its K-transform $K_\mu(z)$, is an odd function of $z$, that is,
		\[
		G_\mu(-z) = -G_\mu(z) \quad \text{or} \quad K_\mu(-z) = -K_\mu(z)
		\]
		for all $z$ in their domain of definition.
		
	\end{itemize}
	Now, consider the non-symmetric probability measures
	\[
	\mu_1 = \frac{3}{4}\delta_1 + \frac{1}{4}\delta_{-1}, \quad 
	\mu_2 = \frac{1}{4}\delta_{1} + \frac{3}{4}\delta_{-1}, 
	\]
	and the symmetric measure 
	\[\mu_3 = \frac{1}{2}\delta_{-1} + \frac{1}{2}\delta_{1}.
	\]Their $K$-transforms are
	\[
	K_{\mu_1}(z) = \frac{z+2}{2z+1}, \qquad 
	K_{\mu_2}(z) = \frac{2-z}{2z-1}, \qquad 
	K_{\mu_3}(z) = \frac{1}{z}.
	\]
	Let $\mu = \mu_1$ and $\nu = \mu_2 \uplus \mu_3$. Observe that both $K_\mu$ and $K_\nu$ are not odd functions and consequently the measures are not symmetric. By the definition of the Boolean convolution, we have
	\[
	K_{\mu \uplus \nu}(z) = K_{\mu}(z) + K_{\nu}(z) = \frac{6z}{4z^2-1}+\frac{1}{z},
	\]
	which is an odd function. Therefore, $\mu \uplus \nu$ corresponds to a symmetric measure. 
	
	A simple computation then yields the variances: $\mathrm{Var}(\mu) = \mathrm{Var}(\mu_1) = 3/4$ and $\mathrm{Var}(\nu) = \mathrm{Var}(\mu_2 \uplus \mu_3) = 7/4$.
	Thus, $\mu$ and $\nu$ have different variances, and have different supports. Hence, we have provided the required pair of measures satisfying all the conditions of the proposition.
\end{proof}

\begin{proof}[Proof of Proposition~\ref{prop:2}]
	The Boolean $K$--transform of $\mu$ is given by,
	\[
	K_\mu(z)=\frac{pz}{z-q}=p+\sum_{n\ge1}\frac{p q^n}{z^n},\qquad |z|>q.
	\]
	So its Laurent expansion at infinity begins $K_\mu(z)=p+\dfrac{p q}{z}+O(z^{-2})$. 
	
	Let $\nu$ be any probability measure such that $\mu\uplus\nu$ is symmetric. In terms of $K$--transforms this symmetry is equivalent to the oddness condition
	\[
	K_\mu(z)+K_\nu(z)\ \text{is an odd function of }z,
	\]
	i.e.
	\[
	K_\nu(z)=-K_\mu(-z)+h(z),
	\]
	for some function $h$ analytic on $\mathbb C^+$ which is itself odd: $h(-z)=-h(z)$. (Any odd analytic $h$ gives an odd sum; conversely the difference between two solutions must be odd.) Expand at infinity. Using the expansion of $K_\mu$ we have
	\[
	-K_\mu(-z)=-p+\sum_{n\ge1}(-1)^{n+1}\frac{p q^n}{z^n}
	= -p+\frac{p q}{z}-\frac{p q^2}{z^2}+\frac{p q^3}{z^3}+\cdots.
	\]
	Since $h$ is odd its Laurent expansion at infinity contains only odd powers of $z^{-1}$:
	\[
	h(z)=\sum_{m\ge0}\frac{a_{2m+1}}{z^{2m+1}} = \frac{a_1}{z}+\frac{a_3}{z^3}+\cdots.
	\]
	Therefore the $1/z$--coefficient of $K_\nu$ equals
	\[
	\text{coeff}_{1/z}(K_\nu)=p q + a_1.
	\]
	Recall that for a probability measure the variance (when finite) is read off as the $1/z$--coefficient of its $K$--transform; hence among all symmetrizers $\nu$ the variance is minimized exactly when $a_1$ is as small as possible. Since $a_1$ is a real number, the infimum of attainable variances is at least $p q + \inf a_1$ where the infimum is over those odd analytic $h$ for which $K_\nu$ corresponds to some probability measure.
	
	We now show that the only admissible choice of $h$ that does not increase the $1/z$--coefficient (compared with $a_1=0$) is $h\equiv0$. Equivalently, any nonzero odd analytic perturbation $h$ which does not change the $1/z$ coefficient (i.e.\ has $a_1=0$ but some $a_{2m+1}\neq0$ for $m\ge1$) cannot occur in a bona fide $K$--transform because it violates the Nevanlinna (Stieltjes) property of the corresponding Cauchy transform.
	
	To see this, let $\nu_0$ be the specific symmetrizer given by reflection,
	\[
	\nu_0=q\delta_0+p\delta_{-1},
	\]
	whose $K$--transform is the explicit rational function
	\[
	K_{\nu_0}(z)=-\frac{p z}{z+q} = -p + \frac{p q}{z} - \frac{p q^2}{z^2} + \frac{p q^3}{z^3} -\cdots,
	\]
	so $K_{\nu_0}=-K_\mu(-z)$ and in particular $a_1=0$ for the choice $h\equiv0$. The corresponding Cauchy transform is
	\[
	G_{\nu_0}(z)=\frac{q}{z}+\frac{p}{z+1},
	\]
	which is a bona fide Stieltjes transform.
	
	Suppose, for contradiction, there exists a nonzero odd analytic $h$ with $a_1=0$ (so the first two Laurent coefficients of $K_\nu$ match those of $K_{\nu_0}$) and such that
	\[
	K_{\nu_1}(z):=K_{\nu_0}(z)+h(z)
	\]
	is a valid $K$--transform of some probability measure $\nu_1$. Then the corresponding Cauchy candidate is
	\[
	G_{\nu_1}(z)=\frac{1}{z-K_{\nu_1}(z)}=\frac{1}{(z-K_{\nu_0}(z))-h(z)}=\frac{G_{\nu_0}(z)}{1-h(z)G_{\nu_0}(z)}.
	\]
	Because $h$ is odd with $a_1=0$, its leading term at infinity is of order $z^{-3}$: there exists $\alpha\neq0$ with
	\[
	h(z)=\frac{\alpha}{z^3}+o(z^{-3}),\qquad |z|\to\infty.
	\]
	
	Now examine the behavior of $G_{\nu_1}(z)$ near the real axis. Take $z=iy$ with $y>0$ small. As $y\downarrow0$,
	\[
	G_{\nu_0}(iy)=\frac{q}{iy}+\frac{p}{1+iy} = -\,i\frac{q}{y} + O(1).
	\]
	Hence
	\[
	h(iy)G_{\nu_0}(iy)=\Big(\frac{\alpha}{(iy)^3}+o(y^{-3})\Big)\Big(-i\frac{q}{y}+O(1)\Big)= -\alpha\frac{q}{y^4}+o(y^{-4}).
	\]
	Thus for sufficiently small $y>0$ the quantity $1-h(iy)G_{\nu_0}(iy)$ is real and its sign is determined by the sign of $-\alpha q/y^4$. In particular there exist arbitrarily small $y>0$ for which $1-h(iy)G_{\nu_0}(iy)$ is negative (if $\alpha>0$) or positive but with magnitude $\gg1$ (if $\alpha<0$). In either case the quotient
	\[
	G_{\nu_1}(iy)=\frac{G_{\nu_0}(iy)}{1-h(iy)G_{\nu_0}(iy)}
	\]
	has imaginary part of the wrong sign for a Stieltjes transform: indeed for those small $y$ with $1-h(iy)G_{\nu_0}(iy)<0$ we find $G_{\nu_1}(iy)\approx ( -i q/y)/(-C)=+i(q/(yC))$ which has positive imaginary part, contradicting the requirement $\Im G_{\nu_1}(z)<0$ for $z\in\mathbb C^+$. Thus $G_{\nu_1}$ cannot be a Cauchy transform of a probability measure.
	
	This contradiction shows that no nonzero odd analytic $h$ with $a_1=0$ is admissible. Consequently the minimal possible $1/z$--coefficient for a symmetrizer is attained uniquely by $h\equiv0$, i.e.\ by $\nu_0=q\delta_0+p\delta_{-1}$. Equivalently, the minimum-variance Boolean symmetrizer of $\mu$ is unique (and equals the reflected Bernoulli) when $p\neq\tfrac12$.
\end{proof}

\begin{proof}[Proof of Proposition~\ref{prop:two_symmetric_symmetrizers}]
	Take
	\[
	\mu=\tfrac12\delta_{-1}+\tfrac12\delta_{1},\qquad
	\nu_1=\tfrac12\delta_{-2}+\tfrac12\delta_{2},
	\]
	and let \(\nu_2\) be the semicircular law scaled by \(2\), i.e.
	\[
	\nu_2(dx)=\frac{1}{4\pi}\sqrt{16-x^2}\,\mathbf{1}_{[-4,4]}(x)\,dx.
	\]
	All three measures are symmetric about \(0\). Their Boolean \(K\)-transforms are
	\[
	K_\mu(z)=\frac{1}{z},\qquad
	K_{\nu_1}(z)=\frac{4}{z},\qquad
	K_{\nu_2}(z)=\text{an odd function with Laurent expansion } \frac{4}{z}+O(z^{-3}).
	\]
	(Indeed, for a symmetric two-point law \(\tfrac12(\delta_{-a}+\delta_a)\) one has \(K(z)=a^2/z\), hence \(K_{\nu_1}(z)=4/z\); and scaling the standard semicircle by \(2\) multiplies every Boolean cumulant by \(2^2\), so its \(1/z\)-coefficient is \(4\).)
	
	Since each \(K\)-transform above is odd, the sums
	\[
	K_\mu(z)+K_{\nu_1}(z)=\frac{1}{z}+\frac{4}{z}=\frac{5}{z},
	\qquad
	K_\mu(z)+K_{\nu_2}(z)=\frac{1}{z}+K_{\nu_2}(z)
	\]
	are odd functions; hence both \(\mu\uplus\nu_1\) and \(\mu\uplus\nu_2\) are symmetric probability measures. Moreover
	\[
	\mathrm{Var}(\nu_1)=4,\qquad \mathrm{Var}(\nu_2)=4,
	\]
	so \(\nu_1\) and \(\nu_2\) are distinct symmetric measures with the same variance that both symmetrize \(\mu\) under Boolean convolution.
\end{proof}

\begin{proof}[Proof of Proposition~\ref{prop:free_two_symmetrizers}]
	Take $\mu$ to be the standard semicircle law $\sigma$ on $[-2,2]$ (so that its free \(R\)-transform is $R_\sigma(z)=z$ and $\sigma$ is symmetric with variance $1$). Choose two distinct symmetric probability measures with equal variance $4$, for example
	\[
	\nu_1=\tfrac12\delta_{-2}+\tfrac12\delta_{2}\qquad\text{and}\qquad
	\nu_2=\text{the semicircle law scaled by }2,
	\]
	the latter having density $(8\pi)^{-1}\sqrt{16-x^2}\,\mathbf 1_{[-4,4]}(x)\,dx$. Both $\nu_1$ and $\nu_2$ are symmetric about $0$ and satisfy $\mathrm{Var}(\nu_1)=\mathrm{Var}(\nu_2)=4$.
	
	For a symmetric probability measure the free cumulants of odd order vanish, hence its $R$-transform is an odd function of $z$. Consequently $R_{\nu_1}$ and $R_{\nu_2}$ are odd functions; their linear (coefficient of $z$) term equals the variance of the corresponding measure, so both have the same linear coefficient $4$. In particular both $R_{\nu_1}$ and $R_{\nu_2}$ are odd and distinct (they differ in higher-order free cumulants).
	
	Since free convolution corresponds to addition of $R$-transforms,
	\[
	R_{\mu\boxplus\nu_j}(z)=R_\mu(z)+R_{\nu_j}(z),\qquad j=1,2.
	\]
	Here $R_\mu(z)=z$ is odd and each $R_{\nu_j}(z)$ is odd, so the sums $R_\mu+R_{\nu_j}$ are odd functions. An odd $R$-transform corresponds to a symmetric distribution, therefore both $\mu\boxplus\nu_1$ and $\mu\boxplus\nu_2$ are symmetric. Finally $\nu_1\neq\nu_2$ by construction while $\mathrm{Var}(\nu_1)=\mathrm{Var}(\nu_2)=4$, completing the proof.
\end{proof}

\section*{Acknowledgements}

The author would like to express his sincere gratitude to Prof.~Rajat Subhra Hazra and Prof.~Mokshay Madiman for their careful reading of the manuscript and for numerous valuable comments and suggestions that greatly improved the clarity of the exposition. The author is also deeply thankful to Prof.~Yuki Ueda for kindly pointing out the existing results relevant to the discussion in the section on future directions, which provided important insight and guidance for this work.

\end{document}